\documentclass[11pt]{amsart}
\usepackage[margin=1in]{geometry}
\usepackage{latexsym}
\usepackage{amsfonts}
\usepackage{amsmath}
\usepackage{amssymb}
\usepackage{amsthm}
\usepackage{enumerate}
\usepackage[hang,flushmargin]{footmisc}
\usepackage{caption}
\usepackage{tabu}
\usepackage{mathrsfs}
\usepackage{amsaddr}

\usepackage{graphicx}
\usepackage{epstopdf}
\usepackage{epsfig}
 
\usepackage{bm} 
\usepackage{tikz}

\newtheorem{theorem}{Theorem}[section]
\newtheorem{lemma}{Lemma}[section]
\newtheorem{proposition}{Proposition}[section]
\newtheorem{corollary}{Corollary}[section]
\newtheorem{fact}{Fact}[section]
\newtheorem{conjecture}{Conjecture}[section]

\theoremstyle{definition}
\newtheorem{definition}{Definition}[section]
\newtheorem{remark}{Remark}[section]

\begin{document}
\title{Anti-power Prefixes of the Thue-Morse Word}
\author{Colin Defant}
\address{University of Florida \\ 1400 Stadium Rd. \\ Gainesville, FL 32611 United States}
\email{cdefant@ufl.edu}

\begin{abstract} 
Recently, Fici, Restivo, Silva, and Zamboni defined a $k$-anti-power to be a word of the form $w_1w_2\cdots w_k$, where $w_1,w_2,\ldots,w_k$ are distinct words of the same length. They defined $AP(x,k)$ to be the set of all positive integers $m$ such that the prefix of length $km$ of the word $x$ is a $k$-anti-power. Let ${\bf t}$ denote the Thue-Morse word, and let $\mathcal F(k)=AP({\bf t},k)\cap(2\mathbb Z^+-1)$. For $k\geq 3$, $\gamma(k)=\min(\mathcal F(k))$ and $\Gamma(k)=\max((2\mathbb Z^+-1)\setminus\mathcal F(k))$ are well-defined odd positive integers. Fici et al. speculated that $\gamma(k)$ grows linearly in $k$. We prove that this is indeed the case by showing that $1/2\leq\displaystyle{\liminf_{k\to\infty}}(\gamma(k)/k)\leq 9/10$ and $1\leq\displaystyle{\limsup_{k\to\infty}}(\gamma(k)/k)\leq 3/2$. In addition, we prove that $\displaystyle{\liminf_{k\to\infty}}(\Gamma(k)/k)=3/2$ and $\displaystyle{\limsup_{k\to\infty}}(\Gamma(k)/k)=3$. 
\end{abstract} 

\maketitle

\bigskip


 
\section{Introduction} 
A well-studied notion in combinatorics on words is that of a $k$-power; this is simply a word of the form $w^k$ for some word $w$. It is often interesting to ask questions related to whether or not certain types of words contain factors (also known as substrings) that are $k$-powers for some fixed $k$. For example, in 1912, Axel Thue \cite{Thue12} introduced an infinite binary word that does not contain any $3$-powers as factors (we say such a word is cube-free). This infinite word is now known as the Thue-Morse word; it is arguably the world's most famous (mathematical) word \cite{Allouche15, Allouche99, Bugeaud14, Brlek89, Dejean72}. 
\begin{definition}\label{Def1}
Let $\overline w$ denote the Boolean complement of a binary word $w$. Let $A_0=0$. For each nonnegative integer $n$, let $B_n=\overline {A_n}$ and $A_{n+1}=A_nB_n$. The \emph{Thue-Morse word} ${\bf t}$ is defined by \[{\bf t}=\lim_{n\to\infty}A_n.\] 
\end{definition}

Recently, Fici, Restivo, Silva, and Zamboni \cite{Fici16} introduced the very natural concept of a $k$-anti-power; this is a word of the form $w_1w_2\cdots w_k$, where $w_1,w_2,\ldots, w_k$ are distinct words \emph{of the same length}. For example, $001011$ is a $3$-anti-power, while $001010$ is not. In \cite{Fici16}, the authors prove that for all positive integers $k$ and $r$, there is a positive integer $N(k,r)$ such that all words of length at least $N(k,r)$ contain a factor that is either a $k$-power or an $r$-anti-power. They also define $AP(x,k)$ to be the set of all positive integers $m$ such that the prefix of length $km$ of the word $x$ is a $k$-anti-power. We will consider this set when $x={\bf t}$ is the Thue-Morse word. It turns out that $AP({\bf t},k)$ is nonempty for all positive integers $k$ \cite[Corollary 6]{Fici16}. It is not difficult to show that if $k$ and $m$ are positive integers, then $m\in AP({\bf t},k)$ if and only if $2m\in AP({\bf t},k)$. Therefore, the only interesting elements of $AP({\bf t},k)$ are those that are odd. For this reason, we make the following definition.  
\begin{definition}\label{Def2}
Let $\mathcal F(k)$ denote the set of odd positive integers $m$ such that the prefix of ${\bf t}$ of length $km$ is a $k$-anti-power. Let $\gamma(k)=\min(\mathcal F(k))$ and $\Gamma(k)=\sup((2\mathbb Z^+-1)\setminus\mathcal F(k))$.  
\end{definition}

\begin{remark}\label{Rem1}
It is immediate from Definition \ref{Def2} that $\mathcal F(1)\supseteq\mathcal F(2)\supseteq \mathcal F(3)\supseteq\cdots$. Therefore, $\gamma(1)\leq\gamma(2)\leq\gamma(3)\leq\cdots$ and $\Gamma(1)\leq\Gamma(2)\leq\Gamma(3)\leq\cdots$. 
\end{remark}

For convenience, we make the following definition. 

\begin{definition}\label{Def3}
If $m$ is a positive integer, let $\mathfrak K(m)$ denote the smallest positive integer $k$ such that the prefix of $\bf t$ of length $km$ is not a $k$-anti-power. 
\end{definition}

If $k\geq 3$, then $(2\mathbb Z^+-1)\setminus\mathcal F(k)$ is nonempty because it contains the number $3$ (the prefix of ${\bf t}$ of length $9$ is $011010011$, which is not a $3$-anti-power). We will show (Theorem \ref{Thm1}) that $(2\mathbb Z^+-1)\setminus\mathcal F(k)$ is finite so that $\Gamma(k)$ is a positive integer for each $k\geq 3$. For example, $(2\mathbb Z^+-1)\setminus\mathcal F(6)=\{1,3,9\}$. This means that $AP({\bf t},6)$ is the set of all postive integers of the form $2^\ell m$, where $\ell$ is a nonnegative integer and $m$ is an odd integer that is not $1$, $3$, or $9$.  

Fici et al. \cite{Fici16} give the first few values of the sequence $\gamma(k)$ and speculate that the sequence grows linearly in $k$. We will prove that this is indeed the case. In fact, it is the aim of this paper to prove the following: 
\begin{itemize}
\item $\displaystyle{\frac 12\leq\liminf_{k\to\infty}\frac{\gamma(k)}{k}\leq \frac{9}{10}}$
\item $\displaystyle{1\leq\limsup_{k\to\infty}\frac{\gamma(k)}{k}\leq\frac 32}$
\item $\displaystyle{\liminf_{k\to\infty}\frac{\Gamma(k)}{k}=\frac{3}{2}}$
\item $\displaystyle{\limsup_{k\to\infty}\frac{\Gamma(k)}{k}=3}$.
\end{itemize}

Despite these asymptotic results, there are many open problems arising from consideration of the sets $\mathcal F(k)$ (such as the cardinality of $(2\mathbb Z^+-1)\setminus\mathcal F(k)$) that we have not investigated; we discuss some of these problems at the end of the paper. 

\section{The Thue-Morse Word: Background and Notation}
Our primary focus is on the Thue-Morse word ${\bf t}$. In this brief section, we discuss some of the basic properties of this word that we will need when proving our asymptotic results. 

Let ${\bf t}_i$ denote the $i^\text{th}$ letter of ${\bf t}$ so that ${\bf t}={\bf t}_1{\bf t}_2{\bf t}_3\cdots$. The number ${\bf t}_i$ has the same parity as the number of $1$'s in the binary expansion of $i-1$. For any positive integers $\alpha,\beta$ with $\alpha\leq \beta$, define $\langle \alpha,\beta\rangle={\bf t}_\alpha{\bf t}_{\alpha+1}\cdots{\bf t}_\beta$. In his seminal 1912 paper, Thue proved that $\bf t$ is overlap-free \cite{Thue12}. This means that if $x$ and $y$ are finite words and $x$ is nonempty, then $xyxyx$ is not a factor of $\bf t$. Equivalently, if $a,b,n$ are positive integers satisfying $a<b\leq a+n$, then $\langle a,a+n\rangle\neq\langle b,b+n\rangle$. Note that this implies that ${\bf t}$ is cube-free.      

We write $\mathbb A^{\leq\omega}$ to denote the set of all words over an alphabet $\mathbb A$. Let $\mathcal W_1$ and $\mathcal W_2$ be sets of words. A \emph{morphism} $f\colon\mathcal W_1\to\mathcal W_2$ is a function satisfying $f(xy)=f(x)f(y)$ for all words $x,y\in\mathcal W_1$. A morphism is uniquely determined by where it sends letters. Let $\mu\colon\{0,1\}^{\leq\omega}\to\{01,10\}^{\leq\omega}$ denote the morphism defined by $\mu(0)=01$ and $\mu(1)=10$. Also, define a morphism $\sigma\colon\{01,10\}^{\leq\omega}\to\{0,1\}^{\leq\omega}$ by $\sigma(01)=0$ and $\sigma(10)=1$ so that $\sigma=\mu^{-1}$. The words ${\bf t}$ and $\overline{\bf t}$ are the unique one-sided infinite words over the alphabet $\{0,1\}$ that are fixed by $\mu$. Because $\mu({\bf t})={\bf t}$, we may view ${\bf t}$ as a word over the alphabet $\{01,10\}$. In particular, this means that ${\bf t}_{2i-1}\neq{\bf t}_{2i}$ for all positive integers $i$. In addition, if $\alpha$ and $\beta$ are nonnegative integers with $\alpha<\beta$, then $\langle 2\alpha+1,2\beta\rangle\in\{01,10\}^{\leq\omega}$. Recall the definitions of $A_n$ and $B_n$ from Definition \ref{Def1}. Observe that $A_n=\mu^n(0)$ and $B_n=\mu^n(1)$. Because $\mu^n({\bf t})={\bf t}$, the Thue-Morse word is actually a word over the alphabet $\{A_n,B_n\}$. This leads us to the following simple but useful fact. 
\begin{fact}\label{Fact1}
For any positive integers $n$ and $r$, $\langle 2^nr+1,2^n(r+1)\rangle=\mu^n(t_{r+1})$.  
\end{fact}

\section{Asymptotics for $\Gamma(k)$} 

In this section, we prove that $\displaystyle{\liminf_{k\to\infty}}\:\Gamma(k)/k=3/2$ and $\displaystyle{\limsup_{k\to\infty}}\:\Gamma(k)/k=3$. The following proposition will prove very useful when we do so. 

\begin{proposition}\label{Prop1}
Let $m\geq 2$ be an integer, and let $\delta(m)=\lceil\log_2 (m/3)\rceil$. 

\begin{enumerate}[(i)]
\item If $y$ and $v$ are words such that $yvy$ is a factor of $\bf t$ and $\vert y\vert=m$, then $2^{\delta(m)}$ divides $\vert yv\vert$. 
\item There is a factor of $\bf t$ of the form $yvy$ such that $\vert y\vert=m$ and $2^{\delta(m)+1}$ does not divide $\vert yv\vert$.
\end{enumerate}  
\end{proposition}\vspace{-.5cm}
\begin{proof}
We first prove $(ii)$ by induction on $m$. If $m=2$, we may simply set $y=01$ and $v=1$. If $m=3$, we may set $y=101$ and $v=\varepsilon$ (the empty word). Now, assume $m\geq 4$. First, suppose $m$ is even. By induction, we can find a factor of $\bf t$ of the form $yvy$ such that $\vert y\vert=m/2$ and such that $2^{\delta(m/2)+1}$ does not divide $\vert yv\vert$. Note that $\mu(y)\mu(v)\mu(y)$ is a factor of $\bf t$ and that $2^{\delta(m/2)+2}$ does not divide $2\vert yv\vert=\vert\mu(y)\mu(v)\vert$. Since $\delta(m/2)+2=\delta(m)+1$, we are done. Now, suppose $m$ is odd. Because $m+1$ is even, we may use the above argument to find a factor $y'v'y'$ of $\bf t$ with $\vert y'\vert=m+1$ such that $2^{\delta(m+1)+1}$ does not divide $\vert y'v'\vert$. It is easy to show that $\delta(m)=\delta(m+1)$ because $m>3$ is odd. This means that $2^{\delta(m)+1}$ does not divide $\vert y'v'\vert$. Let $a$ be the last letter of $y'$, and write $y'=y''a$. Put $v''=av'$. Then $y''v''y''$ is a factor of $\bf t$ with $\vert y''\vert=m$ and $\vert y''v''\vert=\vert y'v'\vert$. This completes the inductive step. 

We now prove $(i)$ by induction on $m$. If $m\leq 3$, the proof is trivial because $\delta(2)=\delta(3)=0$. Therefore, assume $m\geq 4$. Assume that $yvy$ is a factor of $\bf t$ and $\vert y\vert=m$. Let us write ${\bf t}=xyvyz$.    

Suppose by way of contradiction that $\vert vy\vert$ is odd. Then $\vert xy\vert$ and $\vert xyvy\vert$ have different parities. Write $y=y_1a$, where $a$ is the last letter of $y$. Either $xy$ or $xyvy$ is an even-length prefix of $\bf t$, and is therefore a word in $\{01,10\}^{\leq\omega}$. It follows that the second-to-last letter of $y$ is $\overline a$, so we may write $y_1=y_2\overline a$. We now observe that one of the words $xy_1$ and $xyvy_1$ is an even-length prefix of $\bf t$, so the same reasoning as before tells us that the second-to-last letter in $y_1$ is $a$. Therefore, $y=y_3a\overline a a$ for some word $y_3$. We can continue in this fashion to see that $a\overline a a\overline a a$ is a suffix of $vy$. This is impossible since $\bf t$ is overlap-free. Hence, $\vert vy\vert$ must be even. We now consider four cases corresponding to the possible parities of $\vert x\vert$ and $m$. 

\noindent {\bf Case 1:} $\vert x\vert$ and $\vert y\vert=m$ are both even. We just showed $\vert vy\vert$ is even, so all of the words $x,xy,xyv,xyvy$ are even-length prefixes of $\bf t$. This means that $x,y,v,z\in\{01,10\}^{\leq\omega}$, so ${\bf t}=\sigma(x)\sigma(y)\sigma(v)\sigma(y)\sigma(z)$. By induction, we see that $2^{\delta(\vert\sigma(y)\vert)}$ divides $\vert\sigma(y)\sigma(v)\vert$. Because $\delta(\vert \sigma(y)\vert)=\delta(m/2)=\delta(m)-1$ and $\vert\sigma(y)\sigma(v)\vert=\vert yv\vert/2$, it follows that $2^{\delta(m)}$ divides $\vert yv\vert$.   

\noindent {\bf Case 2:} $\vert x\vert$ is odd and $m$ is even. As in the previous case, $\vert v\vert$ must be even. Let $a,b,c$ be the last letters of $y,v,x$, respectively. Write $y=y_0a$, $v=v_0b$, $x=x_0c$. We have ${\bf t}=x_0cy_0av_0by_0az$. Note that $\vert x_0\vert$, $\vert cy_0\vert$, $\vert av_0\vert$, and $\vert by_0\vert$ are all even. In particular, $cy_0$ and $by_0$ are both in $\{01,10\}^{\leq\omega}$. As a consequence, $b=c$. Setting $x'=x_0$, $y'=by_0$, $v'=av_0$, $z'=az$, we find that ${\bf t}=x'y'v'y'z'$. We are now in the same situation as in the previous case because $\vert x'\vert$ is even and $\vert y'\vert=m$. Consequently, $2^{\delta(m)}$ divides $\vert y'v'\vert=\vert yv\vert$. 

\noindent {\bf Case 3:} $m$ is odd and $\vert x\vert$ is even. Let $a$ be the last letter of $y$. Both $v$ and $z$ start with the letter $\overline a$, so we may write $v=\overline av_1$ and $z=\overline az_1$. Put $x_1=x$ and $y_1=y\overline a$. We have ${\bf t}=x_1y_1v_1y_1z_1$. Because $\vert x_1\vert$ and $\vert y_1\vert=m+1$ are both even, we know from the first case that $2^{\delta(m+1)}$ divides $\vert y_1v_1\vert=\vert yv\vert$. Now, simply observe that $\delta(m)=\delta(m+1)$ because $m>3$ is odd.

\noindent {\bf Case 4:} $m$ and $\vert x\vert$ are both odd. Let $d$ be the first letter of $y$. Both $x$ and $v$ end in the letter $\overline d$, so we may write $x=x_2\overline d$ and $v=v_2\overline d$. Let $y_2=\overline dy$ and $z_2=z$. Then ${\bf t}=x_2y_2v_2y_2z_2$. Because $\vert x_2\vert$ and $\vert y_2\vert=m+1$ are both even, we know that $2^{\delta(m+1)}$ divides $\vert y_2v_2\vert=\vert yv\vert$. Again, $\delta(m)=\delta(m+1)$.
\end{proof}
\begin{corollary}\label{Cor1}
Let $m$ be a positive integer, and let $\delta(m)=\lceil\log_2(m/3)\rceil$. If $k\geq 3$ and $m\in(2\mathbb Z^+-1)\setminus\mathcal F(k)$, then $k-1\geq 2^{\delta(m)}$. 
\end{corollary}
\begin{proof}
There exist integers $n_1$ and $n_2$ with $0\leq n_1<n_2\leq k-1$ such that $\langle n_1m+1,(n_1+1)m\rangle=\langle n_2m+1,(n_2+1)m\rangle$. Let $y=\langle n_1m+1,(n_1+1)m\rangle$ and $v=\langle (n_1+1)m+1,n_2m\rangle$. The word $yvy$ is a factor of $\bf t$, and $\vert y\vert=m$. According to Proposition \ref{Prop1}, $2^{\delta(m)}$ divides $\vert yv\vert=(n_2-n_1)m$, where $\delta(m)=\lceil\log_2(m/3)\rceil$. Since $m$ is odd, $2^{\delta(m)}$ divides $n_2-n_1$. This shows that $k-1\geq n_2\geq n_2-n_1\geq 2^{\delta(m)}$. 
\end{proof}

The following lemma is somewhat technical, but it will be useful for constructing specific pairs of identical factors of the Thue-Morse word. These specific pairs of factors will provide us with odd positive integers $m$ for which $\mathfrak K(m)$ is relatively small. We will then make use of the fact, which follows immediately from Definitions \ref{Def2} and \ref{Def3}, that $\Gamma(k)\geq m$ whenever $k\geq \mathfrak K(m)$.  

\begin{lemma}\label{Lem1}
Suppose $r,m,\ell,h,p,q$ are nonnegative integers satisfying the following conditions:  
\begin{itemize}
\item $h<2^{\ell-2}$
\item $rm=p\cdot 2^{\ell+1}+2^{\ell-1}+h$
\item $(r+1)m\leq p\cdot 2^{\ell+1}+5\cdot 2^{\ell-2}$ 
\item $(r+2^{\ell-2})m=q\cdot 2^{\ell+1}+3\cdot2^{\ell-2}+h$ 
\item ${\bf t}_{p+1}\neq{\bf t}_{q+1}$.
\end{itemize} 
Then $\langle rm+1,(r+1)m\rangle=\langle (r+2^{\ell-2})m+1,(r+2^{\ell-2}+1)m\rangle$, and $\mathfrak K(m)\leq r+2^{\ell-2}+1$.  
\end{lemma}
\begin{proof}
Let $u=\langle rm+1,(r+1)m\rangle$ and $v=\langle (r+2^{\ell-2})m+1,(r+2^{\ell-2}+1)m\rangle$. Let us assume ${\bf t}_{p+1}=0$; a similar argument holds if we assume instead that ${\bf t}_{p+1}=1$. According to Fact \ref{Fact1}, \[\langle p\cdot 2^{\ell+1}+1,(p+1)2^{\ell+1}\rangle=A_{\ell+1}=A_{\ell-2}B_{\ell-2}B_{\ell-2}A_{\ell-2}B_{\ell-2}A_{\ell-2}A_{\ell-2}B_{\ell-2}.\] We may now use the first three conditions to see that $B_{\ell-2}A_{\ell-2}B_{\ell-2}=xuy$ for some words $x$ and $y$ such that $\vert x\vert=h$ and $\vert y\vert=p\cdot 2^{\ell+1}+5\cdot2^{\ell-2}-(r+1)m$ (see Figure \ref{Fig1}).

We know from the last condition that ${\bf t}_{q+1}=1$, so \[\langle q\cdot 2^{\ell+1}+1,(q+1)2^{\ell+1}\rangle=B_{\ell+1}=B_{\ell-2}A_{\ell-2}A_{\ell-2}B_{\ell-2}A_{\ell-2}B_{\ell-2}B_{\ell-2}A_{\ell-2}.\] The fourth condition tells us that $B_{\ell-2}A_{\ell-2}B_{\ell-2}=x'vy'$ for some words $x'$ and $y'$ with $\vert x'\vert=h$. We have shown that $xuy=x'vy'$, where $\vert x\vert=\vert x'\vert$ and $\vert u\vert=\vert v\vert$. Hence, $u=v$. It follows that the prefix of $\bf t$ of length $(r+2^{\ell-2}+1)m$ is not a $(r+2^{\ell-2}+1)$-anti-power, so $\mathfrak K(m)\leq r+2^{\ell-2}+1$ by definition. 
\end{proof}

\begin{figure}[t]

\begin{tikzpicture}
\draw (0,0) -- (0,2.4) -- (16.5,2.4) -- (16.5,0) -- (0,0);
\draw (0,0.8) -- (16.5,0.8);
\draw (0,1.6) -- (16.5,1.6);
\draw (5.98,0.8) -- (5.98,2.4); 
\draw (10.52,0.8) -- (10.52,2.4); 
\draw (0.7475,0.8) -- (0.7475,1.6);
\draw (1.495,0) -- (1.495,1.6);
\draw (2.2425,0.8) -- (2.2425,1.6);
\draw (2.99,0.8) -- (2.99,1.6);
\draw (3.7375,0) -- (3.7375,1.6);
\draw (4.485,0.8) -- (4.485,1.6);
\draw (5.2325,0.8) -- (5.2325,1.6);
\draw (11.2675,0.8) -- (11.2675,1.6);
\draw (12.015,0.8) -- (12.015,1.6);
\draw (12.7625,0) -- (12.7625,1.6);
\draw (13.51,0.8) -- (13.51,1.6);
\draw (14.2575,0.8) -- (14.2575,1.6);
\draw (15.005,0) -- (15.005,1.6);
\draw (15.7525,0.8) -- (15.7525,1.6);
\draw (2,0) -- (2,0.8);
\draw (3.1,0) -- (3.1,0.8);
\draw (13.2675,0) -- (13.2675,0.8);
\draw (14.3675,0) -- (14.3675,0.8); 
\node at (2.95,2){$A_{\ell+1}$};
\node at (13.55,2){$B_{\ell+1}$}; 
\node at (0.37375,1.2){\small $A_{\ell-2}$};
\node at (1.12125,1.2){\small $B_{\ell-2}$};
\node at (1.86875,1.2){\small $B_{\ell-2}$};
\node at (2.61625,1.2){\small $A_{\ell-2}$};
\node at (3.36375,1.2){\small $B_{\ell-2}$};
\node at (4.11125,1.2){\small $A_{\ell-2}$};
\node at (4.85875,1.2){\small $A_{\ell-2}$};
\node at (5.60625,1.2){\small $B_{\ell-2}$};
\node at (10.89375,1.2){\small $B_{\ell-2}$};
\node at (11.64125,1.2){\small $A_{\ell-2}$};
\node at (12.38875,1.2){\small $A_{\ell-2}$};
\node at (13.13625,1.2){\small $B_{\ell-2}$};
\node at (13.88375,1.2){\small $A_{\ell-2}$};
\node at (14.63125,1.2){\small $B_{\ell-2}$};
\node at (15.37875,1.2){\small $B_{\ell-2}$};
\node at (16.12625,1.2){\small $A_{\ell-2}$};

\node at (1.735,0.4){$x$};
\node at (2.55,0.4){$u$};
\node at (3.43,0.4){$y$};

\node at (13.03,0.44){$x'$};
\node at (13.82,0.4){$v$};
\node at (14.73,0.44){$y'$};

\end{tikzpicture}\captionof{figure}{An illustration of the proof of Lemma \ref{Lem1}.} \label{Fig1}
\end{figure}

We may now use Lemma \ref{Lem1} and Proposition \ref{Prop1} to prove that $\displaystyle{\limsup_{k\to\infty}}\:\Gamma(k)/k=3$. Recall that if $k\geq 3$, then $\Gamma(k)\geq 3$ because $3\in(2\mathbb Z^+-1)\setminus\mathcal F(k)$. A particular consequence of the following theorem is that $(2\mathbb Z^+-1)\setminus\mathcal F(k)$ is finite. It follows that if $k\geq 3$, then $\Gamma(k)$ is an odd positive integer.  

\begin{theorem}\label{Thm1} 
Let $\Gamma(k)$ be as in Definition \ref{Def2}. For all integers $k\geq 3$, we have $\Gamma(k)\leq 3k-4$. Furthermore, $\displaystyle{\limsup_{k\to\infty}\frac{\Gamma(k)}{k}}=3$. 
\end{theorem}
\begin{proof}
Fix $k\geq 3$, and let $m\in(2\mathbb Z^+-1)\setminus\mathcal F(k)$. If $m\leq 5$, then $m\leq 3k-4$ as desired, so assume $m\geq 7$. By Corollary \ref{Cor1}, $k-1\geq 2^{\delta(m)}$, where $\delta(m)=\lceil\log_2(m/3)\rceil$. Since $m\geq 7$ is odd, $\delta(m)>\log_2(m/3)$. This shows that $k-1\geq 2^{\delta(m)}>m/3$, so $m\leq 3k-4$. Consequently, $\Gamma(k)\leq 3k-4$. 

We now show that $\displaystyle{\limsup_{k\to\infty}\frac{\Gamma(k)}{k}}=3$. For each positive integer $\alpha$, let $k_\alpha=2^{2\alpha}+2^\alpha+2$. Let us fix an integer $\alpha\geq 3$ and set $r=2^\alpha+1$, $m=3\cdot2^{2\alpha}-2^\alpha+1$, $\ell=2\alpha+2$, $h=1$, $p=3\cdot 2^{\alpha-3}$, and $q=3\cdot 2^{2\alpha-3}+2^{\alpha-2}$. One may easily verify that these values of $r,m,\ell,h,p$, and $q$ satisfy the first four of the five conditions listed in Lemma \ref{Lem1}. Recall that the parity of ${\bf t}_i$ is the same as the parity of the number of $1$'s in the binary expansion of $i-1$. The binary expansion of $p$ has exactly two $1$'s, and the binary expansion of $q$ has exactly three $1$'s. Therefore, ${\bf t}_{p+1}=0\neq 1={\bf t}_{q+1}$. This shows that all of the conditions in Lemma \ref{Lem1} are satisfied, so $\mathfrak K(m)\leq r+2^{\ell-2}+1=k_\alpha$. The prefix of $\bf t$ of length $k_\alpha m$ is not a $k_\alpha$-anti-power, so $\Gamma(k_\alpha)\geq m=3\cdot 2^{2\alpha}-2^\alpha+1$. For each $\alpha\geq 3$, \[\frac{\Gamma(k_\alpha)}{k_\alpha}\geq \frac{3\cdot 2^{2\alpha}-2^\alpha+1}{2^{2\alpha}+2^\alpha+2}.\qedhere\] 
\end{proof}

In the preceding proof, we found an increasing sequence of positive integers $(k_\alpha)_{\alpha\geq 3}$ with the property that $\Gamma(k_\alpha)/k_\alpha\to 3$ as $\alpha\to\infty$. It will be useful to have two other sequences with similar properties. This is the content of the following lemma. 
\begin{lemma}\label{Lem2}
For integers $\alpha\geq 3$, $\beta\geq 9$, and $\rho\geq 4$, define \[k_\alpha=2^{2\alpha}+2^\alpha+2,\hspace{.5cm}K_\beta=2^{2\beta+1}+3\cdot2^{\beta+3}+49,\hspace{.5cm}and \hspace{.5cm}\kappa_\rho=2^\rho+2.\] We have \[\Gamma(k_\alpha)\geq 3\cdot 2^{2\alpha}-2^\alpha+1,\hspace{.5cm}\Gamma(K_\beta)\geq 3\cdot 2^{2\beta+1}-2^{\beta-1}+1,\hspace{.5cm}and\hspace{.5cm}\Gamma(\kappa_\rho)\geq 5\cdot 2^{\rho-1}-8\chi(\rho)+1, 
\] where $\displaystyle{\chi(\rho)=\begin{cases} 1, & \mbox{if } \rho\equiv 0\pmod 2; \\ 2, & \mbox{if } \rho\equiv 1\pmod 2. \end{cases}}$
\end{lemma}
\begin{proof}
We already derived the lower bound for $\Gamma(k_\alpha)$ in the proof of Theorem \ref{Thm1}. To prove the lower bound for $\Gamma(K_\beta)$, put $r=3\cdot 2^{\beta+3}+48$, $m=3\cdot 2^{2\beta+1}-2^{\beta-1}+1$, $\ell=2\beta+3$, $h=48$, $p=9\cdot2^\beta+17$, and $q=3\cdot 2^{2\beta-2}+143\cdot2^{\beta-4}+17$. Straightforward calculations show that these choices of $r,m,\ell,h,p$, and $q$ satisfy the first four conditions of Lemma \ref{Lem1}. The binary expansion of $p$ has exactly four $1$'s while that of $q$ has exactly nine $1$'s (it is here that we require $\beta\geq 9$). It follows that ${\bf t}_{p+1}=0\neq 1={\bf t}_{q+1}$, so the final condition in Lemma \ref{Lem1} is also satisfied. The lemma tells us that $\mathfrak K(m)\leq r+2^{\ell-2}+1=K_\beta$, so the prefix of $\bf t$ of length $K_\beta m$ is not a $K_\beta$-anti-power. Hence, $\Gamma(K_\beta)\geq m=3\cdot 2^{2\beta+1}-2^{\beta-1}+1$.

To prove the lower bound for $\kappa_\rho$, we again invoke Lemma \ref{Lem1}. Let $r'=1$, $m'=5\cdot 2^{\rho-1}-8\chi(\rho)+1$, $\ell'=\rho+2$, $h'=2^{\rho-1}-8\chi(\rho)+1$, $p'=0$, and $q'=5\cdot 2^{\rho-4}-\chi(\rho)$. These choices satisfy the first four conditions in Lemma \ref{Lem1}. The binary expansion of $q'$ has an odd number of $1$'s, so ${\bf t}_{p'+1}={\bf t}_1=0\neq 1={\bf t}_{q'+1}$. We now know that $\mathfrak K(m')\leq r'+2^{\ell'-2}+1=\kappa_\rho$, so $\Gamma(\kappa_\rho)\geq m'=5\cdot 2^{\rho-1}-8\chi(\rho)+1$.  
\end{proof}

We now use the sequences $(k_\alpha)_{\alpha\geq 3}$, $(K_\beta)_{\beta\geq 9}$, and $(\kappa_\rho)_{\rho\geq 4}$ to prove that $\displaystyle{\liminf_{k\to\infty}}(\Gamma(k)/k)=3/2$.    

\begin{theorem}\label{Thm2}
Let $\Gamma(k)$ be as in Definition \ref{Def2}. We have $\displaystyle{\liminf_{k\to\infty}\frac{\Gamma(k)}{k}}=\frac{3}{2}$. 
\end{theorem}
\begin{proof}
Let $k\geq 3$ be a positive integer, and let $m=\Gamma(k)$. Put $\delta(m)=\lceil\log_2(m/3)\rceil$. Corollary \ref{Cor1} tells us that $k-1\geq 2^{\delta(m)}$. Suppose $k$ is a power of $2$, say $k=2^\lambda$. Then the inequality $k-1\geq 2^{\delta(m)}$ forces $\delta(m)\leq \lambda-1$. Thus, $m\leq 3\cdot2^{\lambda-1}=\dfrac{3}{2}k$. This shows that $\dfrac{\Gamma(k)}{k}\leq\dfrac{3}{2}$ whenever $k$ is a power of $2$, so $\displaystyle{\liminf_{k\to\infty}\frac{\Gamma(k)}{k}}\leq\frac{3}{2}$.

To prove the reverse inequality, we will make use of Lemma \ref{Lem2}. Recall the definitions of $k_\alpha$, $K_\beta$, $\kappa_\rho$, and $\chi(\rho)$ from that lemma. Fix $k\geq \kappa_{18}$, and put $m=\Gamma(k)$. Because $k\geq \kappa_{18}$, we may use Lemma \ref{Lem2} and the fact that $\Gamma$ is nondecreasing (see Remark \ref{Rem1}) to see that $m=\Gamma(k)\geq\Gamma(\kappa_{18})\geq 5\cdot 2^{17}-7$. Let $\ell=\lceil\log_2 m\rceil$ so that $2^{\ell-1}<m<2^\ell$. Note that $\ell\geq 20$. Let us first assume that $3\cdot 2^{\ell-2}-2^{(\ell-2)/2}<m<2^\ell$. Lemma \ref{Lem2} tells us that $\Gamma(\kappa_{\ell-1})\geq 5\cdot 2^{\ell-2}-8\chi(\ell-1)+1$. We also know that $5\cdot 2^{\ell-2}-8\chi(\ell-1)+1>m$, so $\Gamma(\kappa_{\ell-1})>m$. Because $\Gamma$ is nondecreasing, $\kappa_{\ell-1}>k$. Thus, 
\begin{equation}\label{Eq1}
\frac{\Gamma(k)}{k}>\frac{3\cdot 2^{\ell-2}-2^{(\ell-2)/2}}{\kappa_{\ell-1}}=\frac{3\cdot 2^{\ell-2}-2^{(\ell-2)/2}}{2^{\ell-1}+2}
\end{equation} if $3\cdot 2^{\ell-2}-2^{(\ell-2)/2}<m<2^\ell$.

Next, assume $2^{\ell-1}<m\leq 3\cdot 2^{\ell-2}-2^{(\ell-2)/2}$ and $\ell$ is even. According to Lemma \ref{Lem2}, $\Gamma(k_{(\ell-2)/2})\geq 3\cdot 2^{\ell-2}-2^{(\ell-2)/2}+1>m$. Because $\Gamma$ is nondecreasing, $k<k_{(\ell-2)/2}$. Therefore, 
\begin{equation}\label{Eq2}
\frac{\Gamma(k)}{k}>\frac{2^{\ell-1}}{k_{(\ell-2)/2}}=\frac{2^{\ell-1}}{2^{\ell-2}+2^{(\ell-2)/2}+2}.
\end{equation}   

Finally, suppose $2^{\ell-1}<m\leq 3\cdot 2^{\ell-2}-2^{(\ell-2)/2}$ and $\ell$ is odd. Lemma \ref{Lem2} states that $\Gamma(K_{(\ell-3)/2})\geq 3\cdot 2^{\ell-2}-2^{(\ell-5)/2}+1>m$. We know that $k<K_{(\ell-3)/2}$ because $\Gamma$ is nondecreasing. As a consequence, \begin{equation}\label{Eq3}
\frac{\Gamma(k)}{k}>\frac{2^{\ell-1}}{K_{(\ell-3)/2}}=\frac{2^{\ell-1}}{2^{\ell-2}+3\cdot 2^{(\ell+3)/2}+49}.
\end{equation} The inequalities in \eqref{Eq1}, \eqref{Eq2}, and \eqref{Eq3} show that in all cases, $\displaystyle{\frac{\Gamma(k)}{k}>\frac{3\cdot 2^{\ell-2}-2^{(\ell-2)/2}}{2^{\ell-1}+2}
}$. Because $\ell\to\infty$ as $k\to\infty$ ($\Gamma(k)$ cannot be bounded since we have just shown $\Gamma(k)/k$ is bounded away from $0$), we find that $\displaystyle{\liminf_{k\to\infty}\Gamma(k)/k}\geq 3/2$. 
\end{proof}

\section{Asymptotics for $\gamma(k)$}
Having demonstrated that $\displaystyle{\liminf_{k\to\infty}(\Gamma(k)/k)}=3/2$ and $\displaystyle{\limsup_{k\to\infty}(\Gamma(k)/k)}=3$, we turn our attention to $\gamma(k)$. To begin the analysis, we prove some lemmas that culminate in an upper bound for $\mathfrak K(m)$ for any odd positive integer $m$. It will be useful to keep in mind that if $j$ is a nonnegative integer, then ${\bf t}_{2j}\neq {\bf t}_{2j+1}={\bf t}_{j+1}$ and ${\bf t}_{4j+2}={\bf t}_{4j+3}$.

\begin{lemma}\label{Lem3}
Let $m$ be an odd positive integer, and let $\ell=\lceil\log_2 m\rceil$. If $\mathfrak K(m)>2^\ell+1$, then ${\bf t}_{m+1}{\bf t}_{m+2}=11$ and ${\bf t}_{2m+1}{\bf t}_{2m+2}=10$.
\end{lemma}
\begin{proof}
Let $w_0=\langle 1,m\rangle$, $w_1=\langle 2^{\ell-1}m+1,(2^{\ell-1}+1)m\rangle$, and $w_2=\langle 2^{\ell}m+1,(2^\ell+1)m\rangle$. The words $w_0,w_1,w_2$ must be distinct because $\mathfrak K(m)>2^\ell+1$. For each $n\in\{0,1,2\}$, $w_n$ is a prefix of \\ $\langle nm2^{\ell-1}+1,(nm+2)2^{\ell-1}\rangle=\mu^{\ell-1}({\bf t}_{nm+1}{\bf t}_{nm+2})$. It follows that ${\bf t}_1{\bf t}_2$, ${\bf t}_{m+1}{\bf t}_{m+2}$, and ${\bf t}_{2m+1}{\bf t}_{2m+2}$ are distinct. Since ${\bf t}_1{\bf t}_2=01$ and ${\bf t}_{2m+1}\neq{\bf t}_{2m+2}$, we must have ${\bf t}_{2m+1}{\bf t}_{2m+2}=10$. Now, ${\bf t}_{2m+1}{\bf t}_{2m+2}=\mu({\bf t}_{m+1})$, so ${\bf t}_{m+1}=1$. This forces ${\bf t}_{m+1}{\bf t}_{m+2}=11$. 
\end{proof}

\begin{lemma}\label{Lem4}
Let $m\geq 3$ be an odd integer, and let $\ell=\lceil\log_2 m\rceil$. Suppose there is a positive integer $j$ such that ${\bf t}_j{\bf t}_{j+1}={\bf t}_{m+j}{\bf t}_{m+j+1}$. Then $\mathfrak K(m)<\left(1+\dfrac{j+1}{m}\right)2^\ell$.
\end{lemma}
\begin{proof}
First, observe that 
\begin{equation}\label{Eq4}
\langle 2^\ell(j-1)+1,2^\ell(j+1)\rangle=\mu^\ell({\bf t}_j{\bf t}_{j+1})=\mu^\ell({\bf t}_{m+j}{\bf t}_{m+j+1})=\langle 2^\ell(m+j-1)+1,2^\ell(m+j+1)\rangle.
\end{equation} Because $\vert\langle 2^\ell(j-1)+1,2^\ell(j+1)\rangle\vert=2^{\ell+1}>2m$, there is a nonnegative integer $r$ such that 
\begin{equation}\label{Eq7}
\langle 2^\ell(j-1)+1,2^\ell(j+1)\rangle=w\langle rm+1,(r+1)m\rangle z
\end{equation} for some nonempty words $w$ and $z$. Note that $r+1<\dfrac{2^\ell(j+1)}{m}$. It follows from \eqref{Eq7} that \[2^\ell(m+j-1)+1< 2^\ell m+rm+1<2^\ell m+(r+1)m<2^\ell(m+j+1),\] so \[\langle 2^\ell(m+j-1)+1,2^\ell(m+j+1)\rangle=w'\langle (2^\ell+r)m+1,(2^\ell+r+1)m\rangle z'\] for some nonempty words $w'$ and $z'$. Note that $\vert w'\vert=(2^\ell+r)m-2^\ell(m+j-1)=rm-2^\ell(j-1)=\vert w\vert$. Combining this fact with \eqref{Eq4}, we find that \[\langle rm+1,(r+1)m\rangle=\langle (2^\ell+r)m+1,(2^\ell+r+1)m\rangle.\] Consequently, \[\mathfrak K(m)\leq 2^\ell+r+1<2^\ell+\frac{2^\ell(j+1)}{m}.\] 
\end{proof}

\begin{lemma}\label{Lem5}
Let $m$ be an odd positive integer with $m\not\equiv 1\pmod 8$, and let $\ell=\lceil\log_2 m\rceil$. We have $\mathfrak K(m)<\left(1+\frac{37}{m}\right)2^\ell$. 
\end{lemma}
\begin{proof}
Suppose instead that $\mathfrak K(m)\geq\left(1+\frac{37}{m}\right) 2^\ell$. Let us assume for the moment that $m\not\equiv 29\pmod{32}$. We will obtain a contradiction to Lemma \ref{Lem4} by exhibiting a positive integer $j\leq 36$ such that ${\bf t}_j{\bf t}_{j+1}={\bf t}_{m+j}{\bf t}_{m+j+1}$. Because $\mathfrak K(m)>2^\ell+1$, Lemma \ref{Lem3} tells us that ${\bf t}_{m+1}{\bf t}_{m+2}=11$ and ${\bf t}_{2m+1}{\bf t}_{2m+2}=10$. 

First, assume $m\equiv 3\pmod 4$. We have $\langle m+2,m+5\rangle=\mu^2({\bf t}_{(m+5)/4})$, so either $\langle m+2,m+5\rangle=0110$ or $\langle m+2,m+5\rangle=1001$. Since ${\bf t}_{m+2}=1$, we must have $\langle m+2,m+5\rangle=1001$. This shows that ${\bf t}_{m+4}{\bf t}_{m+5}=01={\bf t}_4{\bf t}_5$, so we may set $j=4$. 

Next, assume $m\equiv 5\pmod 8$. Let $x01^s01$ be the binary expansion of $m$, where $x$ is some (possibly empty) string of $0$'s and $1$'s. As $m\equiv 5\pmod 8$ and $m\not\equiv 29\pmod {32}$, we must have $1\leq s\leq 2$. Because ${\bf t}_{m+1}=1$, the number of $1$'s in the binary expansion of $m$ is odd. This means that the parity of the number of $1$'s in $x$ is the same as the parity of $s$.

Suppose $s=1$. The binary expansion of $m+3$ is the string $x1000$, which contains an even number of $1$'s. As a consequence, ${\bf t}_{m+4}=0$. The binary expansion of $m+4$ is $x1001$, so ${\bf t}_{m+5}=1$. This shows that ${\bf t}_{m+4}{\bf t}_{m+5}=01={\bf t}_4{\bf t}_5$, so we may set $j=4$. 

Suppose that $s=2$ and that $x$ ends in a $0$, say $x=y0$. Note that $y$ contains an even number of $1$'s. The binary expansions of $m+19$ and $m+20$ are $y100000$ and $y100001$, respectively, so ${\bf t}_{m+20}{\bf t}_{m+21}=10={\bf t}_{20}{\bf t}_{21}$. We may set $j=20$ in this case.

Assume now that $s=2$ and that $x$ ends in a $1$. Let us write $x=x'01^{s'}$, where $x'$ is a (possibly empty) binary string. For this last step, we may need to add additional $0$'s to the beginning of $x$. Doing so does not raise any issues because it does not change the number of $1$'s in $x$. The binary expansion of $m$ is $x'01^{s'}01101$. Note that the parity of the number of $1$'s in $x'$ is the same as the parity of $s'$. The binary expansions of $m+19$ and $m+35$ are $x'10^{s'+5}$ and $x'10^{s'}10000$, respectively. If $s'$ is even, then we may put $j=20$ because ${\bf t}_{m+20}{\bf t}_{m+21}=10={\bf t}_{20}{\bf t}_{21}$. If $s'$ is odd, then we may set $j=36$ because ${\bf t}_{m+36}{\bf t}_{m+37}=10={\bf t}_{36}{\bf t}_{37}$.     

We now handle the case in which $m\equiv 29\pmod {32}$. Say $m=32n-3$. Let $b$ be the number of $1$'s in the binary expansion of $n$. The binary expansion of $m+17=32n+14$ has $b+3$ $1$'s. Similarly, the binary expansions of $m+18$, $m+19$, $2m+17$, $2m+18$, and $2m+19$ have $b+4$, $b+1$, $b+3$, $b+2$, and $b+3$ $1$'s, respectively. This means that ${\bf t}_{m+18}{\bf t}_{m+19}{\bf t}_{m+20}={\bf t}_{2m+18}{\bf t}_{2m+19}{\bf t}_{2m+20}$. Therefore, \[\langle (m+17)2^{\ell-1}+1,(m+20)2^{\ell-1}\rangle=\mu^{\ell-1}({\bf t}_{m+18}{\bf t}_{m+19}{\bf t}_{m+20})\]
\begin{equation}\label{Eq5}
=\mu^{\ell-1}({\bf t}_{2m+18}{\bf t}_{2m+19}{\bf t}_{2m+20})=\langle (2m+17)2^{\ell-1}+1,(2m+20)2^{\ell-1}\rangle.
\end{equation} We have $\displaystyle{\bigcup_{r=9}^{17}}\left(\dfrac{17}{2r},\dfrac{10}{r+1}\right)=\left(\dfrac{1}{2},1\right)$, so there exists some $r\in\{9,10,\ldots,17\}$ such that $\dfrac{17}{2r}<\dfrac{m}{2^\ell}<\dfrac{10}{r+1}$. Equivalently, $17\cdot 2^{\ell-1}<rm<(r+1)m<20\cdot 2^{\ell-1}$. It follows that  there are nonempty words $w$ and $z$ such that $\langle (m+17)2^{\ell-1}+1,(m+20)2^{\ell-1}\rangle=w\langle (r+2^{\ell-1})m+1,(r+2^{\ell-1}+1)m\rangle z$. Similarly, there are nonempty words $w'$ and $z'$ such that $\langle (2m+17)2^{\ell-1}+1,(2m+20)2^{\ell-1}\rangle=w'\langle (r+2^{\ell})m+1,(r+2^{\ell}+1)m\rangle z'$. Note that $\vert w\vert=rm-17\cdot 2^{\ell-1}=\vert w'\vert$. Invoking \eqref{Eq5} yields $\langle (r+2^{\ell-1})m+1,(r+2^{\ell-1}+1)m\rangle=\langle (r+2^{\ell})m+1,(r+2^{\ell}+1)m\rangle$. This shows that $\mathfrak K(m)\leq r+2^\ell+1\leq 2^\ell+18$, securing our final contradiction to the assumption that $\mathfrak K(m)\geq \left(1+\frac{37}{m}\right)2^\ell$.   
\end{proof}

\begin{lemma}\label{Lem8}
Let $m$ be an odd positive integer, and let $\ell=\lceil\log_2 m\rceil$. Suppose $m=2^Lh+1$, where $L$ and $h$ are integers with $L\geq 3$ and $h$ odd. We have $\mathfrak K(m)<\left(1+\dfrac{2^{L+1}+4}{m}\right)2^\ell$.  
\end{lemma}
\begin{proof}
\vspace{-.5cm}
Suppose instead that $\mathfrak K(m)\geq\left(1+\dfrac{2^{L+1}+4}{m}\right)2^\ell$. We will obtain a contradiction to Lemma \ref{Lem4} by finding a positive integer $j\leq 2^{L+1}+3$ satisfying ${\bf t}_{j}{\bf t}_{j+1}={\bf t}_{m+j}{\bf t}_{m+j+1}$. Let $x01^s0^{L-1}1$ be the binary expansion of $m$, and note that $s\geq 1$. Let $N$ be the number of $1$'s in $x$. The binary expansions of $m+2^L+2$, $m+2^L+3$, $m+2^{L+1}+2$, and $m+2^{L+1}+3$ are $x10^{s+L-2}11$, $x10^{s+L-3}100$, $x10^{s-1}10^{L-2}11$, and $x10^{s-1}10^{L-3}100$. This shows that ${\bf t}_{m+2^L+3}{\bf t}_{m+2^L+4}=10$ if $N$ is even and ${\bf t}_{m+2^{L+1}+3}{\bf t}_{m+2^{L+1}+4}=10$ if $N$ is odd. Observe that ${\bf t}_{2^L+3}{\bf t}_{2^L+4}={\bf t}_{2^{L+1}+3}{\bf t}_{2^{L+1}+4}=10$. Therefore, we may put $j=2^L+3$ if $N$ is even and $j=2^{L+1}+3$ if $N$ is odd.   
\end{proof}

\begin{lemma}\label{Lem6}
Let $m$ be an odd positive integer, and let $\ell=\lceil\log_2 m\rceil$. Assume $m=2^Lh+1$ for some integers $L$ and $h$ with $L\geq 3$ and $h$ odd. If $n$ is an integer such that $2\leq n\leq 2^{L-1}$, ${\bf t}_{m-n}={\bf t}_{m-n+1}$, and $m\leq\left(1-\frac{1}{2n+2}\right)2^\ell$, then $\mathfrak K(m)\leq 2^\ell-n$. 
\end{lemma}
\begin{proof}
\vspace{-.5cm} 
Let $y$ and $z$ be the binary expansions of $2^{L-1}-n$ and $2^{L-1}-n+1$, respectively. If necessary, let the strings $y$ and $z$ begin with additional $0$'s so that $\vert y\vert=\vert z\vert=L-1$. Let $x10^L$ be the binary expansion of $m-1$. The binary expansions of $m-2n-1$ and $2m-2n-1$ are $x0y0$ and $x01y1$, respectively. The quantities of $1$'s in these strings are of the same parity, so ${\bf t}_{m-2n}={\bf t}_{2m-2n}$. Similarly, ${\bf t}_{m-2n+2}={\bf t}_{2m-2n+2}$ because the binary expansions of $m-2n+1$ and $2m-2n+1$ are $x0z0$ and $x01z1$, respectively. Let $a={\bf t}_{m-n}$. Because ${\bf t}_{m-n}={\bf t}_{m-n+1}$ by hypothesis, we have ${\bf t}_{2m-2n}={\bf t}_{2m-2n+2}=\overline{a}$. Therefore, ${\bf t}_{m-2n}={\bf t}_{m-2n+2}=\overline{a}$. The word $\bf t$ is cube-free, so ${\bf t}_{m-2n}{\bf t}_{m-2n+1}{\bf t}_{m-2n+2}=\overline aa\overline a={\bf t}_{2m-2n}{\bf t}_{2m-2n+1}{\bf t}_{2m-2n+2}$. Hence, 
\[\langle (m-2n-1)2^{\ell-1}+1,(m-2n+2)2^{\ell-1}\rangle=\mu^{\ell-1}({\bf t}_{m-2n}{\bf t}_{m-2n+1}{\bf t}_{m-2n+2})\]
\begin{equation}\label{Eq6}
=\mu^{\ell-1}({\bf t}_{2m-2n}{\bf t}_{2m-2n+1}{\bf t}_{2m-2n+2})=\langle (2m-2n-1)2^{\ell-1}+1,(2m-2n+2)2^{\ell-1}\rangle.
\end{equation}  

Now, $m\in\left(2^{\ell-1},\left(1-\dfrac{1}{2n+2}\right)2^\ell\right]\displaystyle{\subseteq\bigcup_{r=n}^{2n-1}\left[\dfrac{2n-2}{r}2^{\ell-1},\dfrac{2n+1}{r+1}2^{\ell-1}\right]}$, so there is some $r\in$ \\$\{n,n+1,\ldots,2n-1\}$ such that $\dfrac{2n-2}{r}2^{\ell-1}\leq m\leq\dfrac{2n+1}{r+1}2^{\ell-1}$. Equivalently, $(m-2n-1)2^{\ell-1}\leq(2^{\ell-1}-r-1)m< (2^{\ell-1}-r)m\leq(m-2n+2)2^{\ell-1}$. We find that \[\langle (m-2n-1)2^{\ell-1}+1,(m-2n+2)2^{\ell-1}\rangle=w\langle (2^{\ell-1}-r-1)m+1,(2^{\ell-1}-r)m\rangle z\] and \[\langle (2m-2n-1)2^{\ell-1}+1,(2m-2n+2)2^{\ell-1}\rangle=w'\langle (2^{\ell}-r-1)m+1,(2^{\ell}-r)m\rangle z'\] for some words $w,w',z,z'$. Because $\vert w\vert=(2n+1)2^{\ell-1}-(r+1)m=\vert w'\vert$, we may use \eqref{Eq6} to deduce that \[\langle (2^{\ell-1}-r-1)m+1,(2^{\ell-1}-r)m\rangle=\langle (2^{\ell}-r-1)m+1,(2^{\ell}-r)m\rangle.\] This shows that $\mathfrak K(m)\leq 2^\ell-r\leq 2^\ell-n$ as desired. 
\end{proof}

\begin{lemma}\label{Lem7}
If $m$ is an odd positive integer and $\ell=\lceil\log_2 m\rceil$, then $\mathfrak K(m)<2^\ell+2^{(\ell+5)/2}+10$. 
\end{lemma}
\begin{proof}
We will assume that $m\geq 65$ (so $\ell\geq 7$). One may easily use a computer to check that the desired result holds when $m<65$. 

If $m\not\equiv 1\pmod 8$, then Lemma \ref{Lem5} tells us that \[\mathfrak K(m)<\left(1+\frac{37}{m}\right)2^\ell<2^\ell+74\leq 2^\ell+2^{(\ell+5)/2}+10.\] Suppose that $m\equiv 1\pmod 8$, and let $m=2^Lh+1$, where $L\geq 3$ and $h$ is odd. First, assume $m>\left(1-\dfrac{1}{2^{L}-4}\right)2^\ell$. Because $2^L\vert 2^\ell-m+1$ and $2^\ell-m+1>0$, we have $2^L\leq 2^\ell-m+1<\dfrac{2^\ell}{2^L-4}+1$. This implies that $2^{2L}-4\cdot2^L< 2^{\ell}+2^L-4$, so $2^L<2^{\ell-L}+5-4\cdot2^{-L}<2^{\ell-L+2}$. Hence,
$L\leq \dfrac{\ell+1}{2}$. By Lemma \ref{Lem8}, \[\mathfrak K(m)<\left(1+\dfrac{2^{L+1}+4}{m}\right)2^\ell<2^\ell+2^{L+2}+8< 2^\ell+2^{(\ell+5)/2}+10.\]      

Next, assume $m\leq\left(1-\dfrac{1}{2^{L}-4}\right)2^\ell$ and $L\geq 4$. Let $n$ be the largest integer satisfying $m-n\equiv 2\pmod 4$ and $n\leq 2^{L-1}$. Note that $m\leq \left(1-\dfrac{1}{2n+2}\right)2^\ell$ because $n\geq 2^{L-1}-3$. As $m-n\equiv 2\pmod 4$, we have ${\bf t}_{m-n}={\bf t}_{m-n+1}$. We have shown that $n$ satisfies the criteria specified in Lemma \ref{Lem6}, so $\mathfrak K(m)\leq 2^\ell-n<2^\ell+2^{(\ell+5)/2}+10$.

Finally, if $L=3$, then Lemma \ref{Lem8} tells us that \[\mathfrak K(m)<\left(1+\frac{20}{m}\right)2^\ell<2^\ell+40<2^\ell+2^{(\ell+5)/2}+10.\qedhere\]  
\end{proof}

At last, we are in a position to prove lower bounds for $\displaystyle{\liminf_{k\to\infty}(\gamma(k)/k)}$ and $\displaystyle{\limsup_{k\to\infty}(\gamma(k)/k)}$. 

\begin{theorem}\label{Thm3}
Let $\gamma(k)$ be as in Definition \ref{Def2}. We have \[\liminf_{k\to\infty}\frac{\gamma(k)}{k}\geq \frac{1}{2}\hspace{.5cm}\text{and}\hspace{.5cm}\limsup_{k\to\infty}\frac{\gamma(k)}{k}\geq 1.\]
\end{theorem}
\begin{proof}
For each positive integer $\ell$, let $g(\ell)=\lfloor2^\ell+2^{(\ell+5)/2}+10\rfloor+1$. Lemma \ref{Lem7} implies that $\mathfrak K(m)<g(\ell)$ for all odd positive integers $m<2^\ell$. It follows from the definition of $\gamma$ that $\gamma(g(\ell))\geq 2^\ell+1$. Therefore, \[\limsup_{k\to\infty}\frac{\gamma(k)}{k}\geq\limsup_{\ell\to\infty}\frac{\gamma(g(\ell))}{g(\ell)}\geq\lim_{\ell\to\infty}\frac{2^\ell+1}{2^\ell+2^{(\ell+5)/2}+11}=1.\]

Now, choose an arbitrary positive integer $k$, and let $\ell=\lceil\log_2(\gamma(k))\rceil$. By the definition of $\gamma$, $k<\mathfrak K(\gamma(k))$. We may use Lemma \ref{Lem7} to find that \[\frac{\gamma(k)}{k}>\frac{\gamma(k)}{2^\ell+2^{(\ell+5)/2}+10}>\frac{2^{\ell-1}}{2^\ell+2^{(\ell+5)/2}+10}.\] Note that this implies that $\gamma(k)\to\infty$ as $k\to\infty$. It follows that $\ell\to\infty$ as $k\to\infty$, so \[\liminf_{k\to\infty}\frac{\gamma(k)}{k}\geq\lim_{\ell\to\infty}\frac{2^{\ell-1}}{2^\ell+2^{(\ell+5)/2}+10}=\frac{1}{2}.\]  
\end{proof}

In our final theorem, we provide upper bounds for $\displaystyle{\liminf_{k\to\infty}(\gamma(k)/k)}$ and $\displaystyle{\limsup_{k\to\infty}(\gamma(k)/k)}$. This will complete our proof of all the asymptotic results mentioned in the introduction. Before proving this theorem, we need one lemma. In what follows, recall that the Thue-Morse word ${\bf t}$ is overlap-free. This means that if $a,b,n$ are positive integers satisfying $a<b\leq a+n$, then $\langle a,a+n\rangle\neq\langle b,b+n\rangle$.  

\begin{lemma}\label{Lem9}
For each integer $\ell\geq 3$, we have \[\mathfrak K(3\cdot 2^{\ell-2}+1)>\frac{5\cdot 2^{2\ell-3}}{3\cdot 2^{\ell-2}+1}\hspace{.5cm}\text{and}\hspace{.5cm}\mathfrak K(2^{\ell-1}+3)>\frac{2^{2\ell-2}}{2^{\ell-1}+3} .\]  
\end{lemma}
\begin{proof}
Fix $\ell\geq 3$, and let $m=3\cdot 2^{\ell-2}+1$ and $m'=2^{\ell-1}+3$. By the definitions of $\mathfrak K(m)$ and $\mathfrak K(m')$, there are nonnegative integers $r<\mathfrak K(m)-1$ and $r'<\mathfrak K(m')-1$ such that $\langle rm+1,(r+1)m \rangle=\langle(\mathfrak K(m)-1)m+1,\mathfrak K(m)m\rangle$ and $\langle r'm'+1,(r'+1)m' \rangle=\langle(\mathfrak K(m')-1)m'+1,\mathfrak K(m')m'\rangle$. According to Proposition \ref{Prop1}, $2^{\ell-1}$ divides $(\mathfrak K(m)-1)m-rm$ and $2^{\ell-2}$ divides $(\mathfrak K(m')-1)m'-r'm'$. Since $m$ and $m'$ are odd, we know that $2^{\ell-1}$ divides $\mathfrak K(m)-r-1$ and $2^{\ell-2}$ divides $\mathfrak K(m')-r'-1$. If $\mathfrak K(m)-r-1\geq 2^\ell$, then $\mathfrak K(m)> \dfrac{5\cdot 2^{2\ell-3}}{3\cdot 2^{\ell-2}+1}$ as desired. Therefore, we may assume $\mathfrak K(m)=r+2^{\ell-1}+1$. By the same token, we may assume that $\mathfrak K(m')=r'+2^{\ell-2}+1$. 

With the aim of finding a contradiction, let us assume $\mathfrak K(m)\leq\dfrac{5\cdot 2^{2\ell-3}}{m}$. Put \[u=\langle rm+1,(r+1)m\rangle\hspace{.5cm}\text{snd}\hspace{.5cm}v=\langle(\mathfrak K(m)-1)m+1,\mathfrak K(m)m\rangle.\] We have \[\mu^{2\ell-3}(01)=\mu^{2\ell-3}({\bf t}_4{\bf t}_5)=\langle3\cdot 2^{2\ell-3}+1,5\cdot 2^{2\ell-3}\rangle=wvz\] for some words $w$ and $z$. Observe that $\vert w\vert=(\mathfrak K(m)-1)m-3\cdot 2^{2\ell-3}=rm+2^{\ell-1}$. Since $\mu^{2\ell-3}(01)=\mu^{2\ell-3}({\bf t}_1{\bf t}_2)=\langle 1,2^{2\ell-3}\rangle$, we have $v=\langle rm+2^{\ell-1}+1,(r+1)m+2^{\ell-1}\rangle$. If we set $a=rm+1$ and $b=rm+2^{\ell-1}+1$, then $a<b\leq a+m$. It follows from the fact that ${\bf t}$ is overlap-free that $u\neq v$. This is a contradiction. 

Assume now that $\mathfrak K(m')\leq\dfrac{2^{2\ell-2}}{m'}$. Let \[u'=\langle r'm'+1,(r'+1)m'\rangle\hspace{.5cm}\text{and}\hspace{.5cm}v'=\langle(\mathfrak K(m')-1)m'+1,\mathfrak K(m')m'\rangle.\] Let $q=\lceil(r'm'+1)/2^{\ell-2}\rceil$ and $H=\min\{(r'+1)m',(q+2)2^{\ell-2}\}$. Finally, put $U=\langle r'm'+1,H\rangle$ and $V=\langle (r'+2^{\ell-2})m'+1,H+2^{\ell-2}m'\rangle$. The word $U$ is the prefix of $u'$ of length $H-r'm'$. Because $\mathfrak K(m')=r'+2^{\ell-2}+1$, $V$ is the prefix of $v'$ of length $H-r'm'$. Since $u'=v'$, we must have $U=V$. 

There are words $w'$ and $z'$ such that \[\mu^{\ell-2}({\bf t}_{q}{\bf t}_{q+1}{\bf t}_{q+2})=\langle (q-1)2^{\ell-2}+1,(q+2)2^{\ell-2}\rangle=w'Uz'.\] Furthermore, \[\mu^{\ell-2}({\bf t}_{q+m'}{\bf t}_{q+m'+1}{\bf t}_{q+m'+2})=\langle (q+m'-1)2^{\ell-2}+1,(q+m'+2)2^{\ell-2}\rangle=w''Vz''\] for some words $w''$ and $z''$. Note that $0\leq\vert w'\vert=r'm'-(q-1)2^{\ell-2}=\vert w''\vert<2^{\ell-2}$ (the inequalities follow from the definition of $q$). The suffix of $\mu^{\ell-2}({\bf t}_q)$ of length $2^{\ell-2}-\vert w'\vert$ is a prefix of $U$. Similarly, the suffix of $\mu^{\ell-2}({\bf t}_{q+m'})$ of length $2^{\ell-2}-\vert w''\vert$ is a prefix of $V$. Since $\vert w'\vert=\vert w''\vert$ and $U=V$, we must have ${\bf t}_q={\bf t}_{q+m'}$. Similar arguments show that ${\bf t}_{q+1}={\bf t}_{q+m'+1}$ and ${\bf t}_{q+2}={\bf t}_{q+m'+2}$ (see Figure \ref{Fig2}). 

Now,\[r'=\mathfrak K(m')-2^{\ell-2}-1\leq\frac{2^{2\ell-2}}{m'}-2^{\ell-2}-1=\frac{2^{2\ell-3}-5\cdot 2^{\ell-2}-3}{m'},\] so $\dfrac{r'm'+1}{2^{\ell-2}}<2^{\ell-1}-5$. Therefore, $q+4<2^{\ell-1}$. It follows that for each $j\in\{0,1,2\}$, the binary expansion of $q+m'+j-1$ has exactly one more $1$ than the binary expansion of $q+j+2$. We find that ${\bf t}_{q+3}{\bf t}_{q+4}{\bf t}_{q+5}=\overline{{\bf t}_{q+m'}{\bf t}_{q+m'+1}{\bf t}_{q+m'+2}}=\overline{{\bf t}_{q}{\bf t}_{q+1}{\bf t}_{q+2}}$. However, utilizing the fact that ${\bf t}$ is cube-free, it is easy to check that $X\overline X$ is not a factor of ${\bf t}$ whenever $X$ is a word of length $3$. This yields a contradiction when we set $X=\overline{{\bf t}_{q}{\bf t}_{q+1}{\bf t}_{q+2}}$.   
\end{proof}

\begin{figure}[t]
\begin{tikzpicture}
\draw (0,0) -- (0,2.4) -- (16.5,2.4) -- (16.5,0) -- (0,0);
\draw (0,0.8) -- (16.5,0.8);
\draw (0,1.6) -- (16.5,1.6);
\draw (5.98,0) -- (5.98,2.4); 
\draw (10.52,0) -- (10.52,2.4);

\draw (1.993,0.8) -- (1.993,1.6);
\draw (3.986,0.8) -- (3.986,1.6); 

\draw (12.513,0.8) -- (12.513,1.6);
\draw (14.506,0.8) -- (14.506,1.6); 

\draw (1.1,0) -- (1.1,0.8);
\draw (5.1,0) -- (5.1,0.8); 

\draw (11.62,0) -- (11.62,0.8);
\draw (15.62,0) -- (15.62,0.8); 

\node at (2.999,2){$\langle(q-1)2^{\ell-2}+1,(q+2)2^{\ell-2}\rangle$};
\node at (1.02,1.2){$\mu^{\ell-2}({\bf t}_q)$};
\node at (3,1.2){$\mu^{\ell-2}({\bf t}_{q+1})$};
\node at (5,1.2){$\mu^{\ell-2}({\bf t}_{q+2})$};
\node at (0.575,0.43){$w'$};
\node at (3.08,0.4){$U$};
\node at (5.55,0.43){$z'$};

\node at (13.55,2){\footnotesize $\langle(q+m'-1)2^{\ell-2}+1,(q+m'+2)2^{\ell-2}\rangle$};
\node at (11.55,1.2){\tiny $\mu^{\ell-2}({\bf t}_{q+m'})$};
\node at (13.55,1.2){\tiny $\mu^{\ell-2}({\bf t}_{q+m'+1})$};
\node at (15.54,1.2){\tiny $\mu^{\ell-2}({\bf t}_{q+m'+2})$};
\node at (11.1,0.43){$w''$};
\node at (13.7,0.4){$V$};
\node at (16.1,0.43){$z''$};
\end{tikzpicture}
\captionof{figure}{An illustration of the proof of Lemma \ref{Lem9}.} \label{Fig2}
\end{figure}

\begin{theorem}\label{Thm4}
Let $\gamma(k)$ be as in Definition \ref{Def2}. We have \[\liminf_{k\to\infty}\frac{\gamma(k)}{k}\leq \frac{9}{10}\hspace{.5cm}\text{and}\hspace{.5cm}\limsup_{k\to\infty}\frac{\gamma(k)}{k}\leq\frac{3}{2}.\]
\end{theorem}
\begin{proof}
For each positive integer $\ell$, let $f(\ell)=\left\lfloor\dfrac{5\cdot 2^{2\ell-3}}{3\cdot 2^{\ell-2}+1}\right\rfloor$ and $h(\ell)=\left\lfloor\dfrac{2^{2\ell-2}}{2^{\ell-1}+3}\right\rfloor$. One may easily verify that $h(\ell)<f(\ell)\leq h(\ell+1)$ for all $\ell\geq 3$. Lemma \ref{Lem9} informs us that $\mathfrak K(3\cdot 2^{\ell-2}+1)>f(\ell)$. This means that the prefix of ${\bf t}$ of length $(3\cdot 2^{\ell-2}+1)f(\ell)$ is an $f(\ell)$-anti-power, so $\gamma(f(\ell))\leq 3\cdot 2^{\ell-2}+1$. As a consequence, \[\liminf_{k\to\infty}\frac{\gamma(k)}{k}\leq\liminf_{\ell\to\infty}\frac{\gamma(f(\ell))}{f(\ell)}\leq\lim_{\ell\to\infty}\frac{3\cdot 2^{\ell-2}+1}{f(\ell)}=\frac{9}{10}.\] 

Now, choose an arbitrary integer $k\geq 3$.  If $h(\ell)<k\leq f(\ell)$ for some integer $\ell\geq 3$, then the prefix of ${\bf t}$ of length $(3\cdot 2^{\ell-2}+1)f(\ell)$ is an $f(\ell)$-anti-power. This implies that $\gamma(k)\leq 3\cdot 2^{\ell-2}+1$, so \[\frac{\gamma(k)}{k}<\frac{3\cdot 2^{\ell-2}+1}{h(\ell)}.\] Alternatively, we could have $f(\ell)<k\leq h(\ell+1)$ for some $\ell\geq 3$. In this case, Lemma \ref{Lem9} tells us that the prefix of ${\bf t}$ of length $(2^{\ell}+3)h(\ell+1)$ is an $h(\ell+1)$-anti-power. It follows that \[\frac{\gamma(k)}{k}<\frac{2^{\ell}+3}{f(\ell)}\] in this case. 

Combining the above cases, we deduce that \[\limsup_{k\to\infty}\frac{\gamma(k)}{k}\leq\limsup_{\ell\to\infty}\left[\max\left\{\frac{3\cdot 2^{\ell-2}+1}{h(\ell)},\frac{2^{\ell+1}+3}{f(\ell)}\right\}\right]=\max\left\{\frac{3}{2},\frac{6}{5}\right\}=\frac{3}{2}.\] 
\end{proof}
\begin{remark}
Preserve the notation from the proof of Theorem \ref{Thm4}. We showed that \[\frac{\gamma(k)}{k}<\frac{3\cdot 2^{\ell-2}+1}{h(\ell)}=\frac{3}{2}+o(1)\] if $h(\ell)<k\leq f(\ell)$ and \[\frac{\gamma(k)}{k}<\frac{2^{\ell}+3}{f(\ell)}=\frac{6}{5}+o(1)\] whenever $f(\ell)<k\leq h(\ell+1)$ (the $o(1)$ terms refer to asymptotics as $k\to\infty$). This is indeed reflected in the top image of Figure \ref{Fig3}, which portrays a plot of $\gamma(k)/k$ for $3\leq k\leq 2100$. 
\end{remark}

\section{Concluding Remarks}
In Theorems \ref{Thm1} and \ref{Thm2}, we obtained the exact values of $\displaystyle{\liminf_{k\to\infty}(\Gamma(k)/k)}$ and $\displaystyle{\limsup_{k\to\infty}(\Gamma(k)/k)}$. Unfortunately, we were not able to determine the exact values of $\displaystyle{\liminf_{k\to\infty}(\gamma(k)/k)}$ and $\displaystyle{\limsup_{k\to\infty}(\gamma(k)/k)}$. Figure \ref{Fig3} suggests that the upper bounds we obtained are the correct values.
\begin{conjecture}\label{Conj1}
We have \[\liminf_{k\to\infty}\frac{\gamma(k)}{k}=\frac{9}{10}\hspace{.5cm}and\hspace{.5cm}\limsup_{k\to\infty}\frac{\gamma(k)}{k}=\frac{3}{2}.\]
\end{conjecture}
Recall that we obtained lower bounds for $\displaystyle{\liminf_{k\to\infty}(\gamma(k)/k)}$ and $\displaystyle{\limsup_{k\to\infty}(\gamma(k)/k)}$ by first showing that $\mathfrak K(m)\leq 2^{\lceil\log_2 m\rceil}(1+o(m))$. If Conjecture \ref{Conj1} is true, its proof will most likely require a stronger upper bound for $\mathfrak K(m)$. 

We know from Theorem \ref{Thm1} that $(2\mathbb Z^+-1)\setminus\mathcal F(k)$ is finite whenever $k\geq 3$. A very natural problem that we have not attempted to investigate is that of determining the cardinality of this finite set. Similarly, one might wish to explore the sequence $(\Gamma(k)-\gamma(k))_{k\geq 3}$. 

Recall that if $w$ is an infinite word whose $i^\text{th}$ letter is $w_i$, then $AP(w,k)$ is the set of all positive integers $m$ such that $w_1w_2\cdots w_{km}$ is a $k$-anti-power. An obvious generalization would be to define $AP_j(w,k)$ to be the set of all positive integers $m$ such that $w_{j+1}w_{j+2}\cdots w_{j+km}$ is a $k$-anti-power. Of course, we would be particularly interested in analyzing the sets $AP_j({\bf t},k)$. 

Define a $(k,\lambda)$-anti-power to be a word of the form $w_1w_2\cdots w_k$, where $w_1,w_2,\ldots,w_k$ are words of the same length and $\vert\{i\in\{1,2,\ldots,k\}\colon w_i=w_j\}\vert\leq\lambda$ for each fixed $j\in\{1,2,\ldots,k\}$. With this definition, a $(k,1)$-anti-power is simply a $k$-anti-power. Let $\mathfrak K_\lambda(m)$ be the smallest positive integer $k$ such that the prefix of ${\bf t}$ of length $km$ is not a $(k,\lambda)$-anti-power. What can we say about $\mathfrak K_\lambda(m)$ for various positive integers $\lambda$ and $m$?   

Finally, note that we may ask questions similar to the ones asked here for other infinite words. In particular, it would be interesting to know other nontrivial examples of infinite words $x$ such that $\min AP(x,k)$ grows linearly in $k$.   

\begin{figure}[t]
\begin{center}
\includegraphics[width=\linewidth]{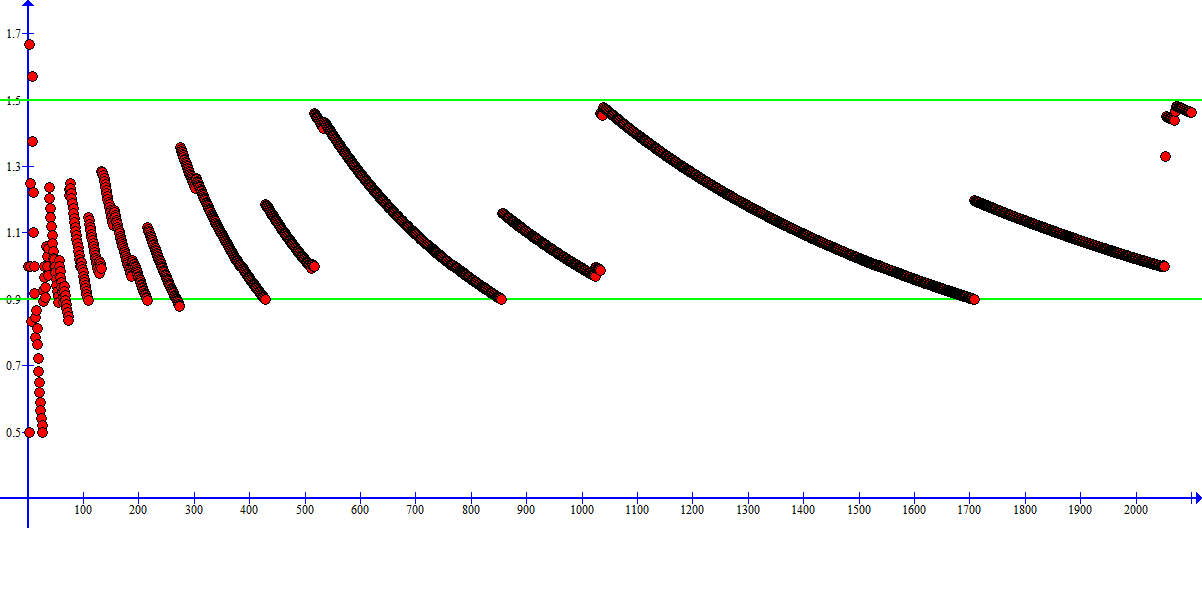}
\includegraphics[width=\linewidth]{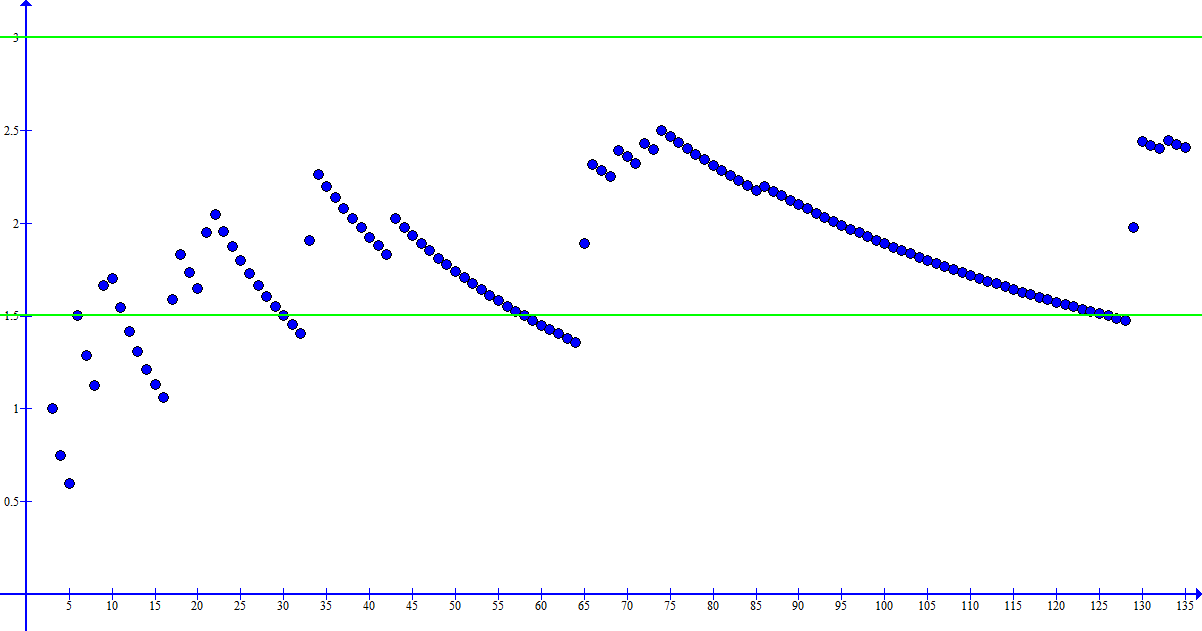}
\end{center}
\captionof{figure}{Plots of $\gamma(k)/k$ for $3\leq k\leq 2100$ (top) and $\Gamma(k)/k$ for $3\leq k\leq 135$ (bottom). In the top image, the green lines are at $y=9/10$ and $y=3/2$. In the bottom image, the green lines are at $y=3/2$ and $y=3$.} \label{Fig3}
\end{figure}

\begin{figure}[t]
\begin{center}
\includegraphics[width=\linewidth]{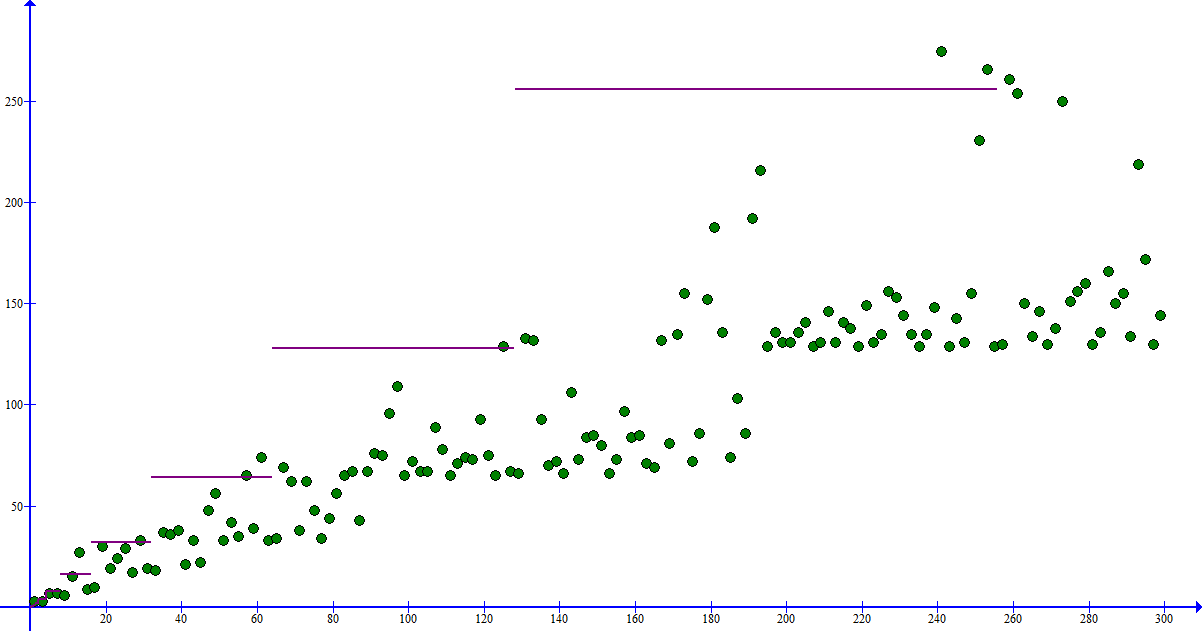}
\end{center}
\captionof{figure}{A plot of $\mathfrak K(m)$ for all odd positive integers $m\leq 299$. In purple is the graph of $y=2^{\lceil\log_2 x\rceil}$.} \label{Fig4}
\end{figure}

\section{Acknowledgments}
This work was supported by the grants NSF-1358659 and NSA H98230-16-1-0026. 

The author would like to thank Jacob Ford for writing a program that produced the values of $\gamma(k)$ for $3\leq k\leq 2100$. The author would also like to thank Joe Gallian and Sam Elder for reading thoroughly through this manuscript and providing helpful feedback.

\end{document}